\documentclass[10pt]{amsart}
\usepackage{amssymb}
\newtheorem{define}{Definition}

\newtheorem{problem}{Problem}
\newtheorem{lemma}{Lemma}
\newtheorem{coro}{Corollary}
\newtheorem{theorem}{Theorem}
\newtheorem{remark}{Remark}

\begin{document}

\title [Functional inequalities for Fox-Wright functions ]{Functional inequalities for Fox-Wright functions \\}%

\author[ K. Mehrez, S. M. Sitnik]{KHALED MEHREZ and Sergei M. Sitnik}
\address{Khaled Mehrez. D\'epartement de Math\'ematiques ISSAT Kasserine, Universit\'e de Kairouan, Tunisia.}
\address{ D\'epartement de Math\'ematiques, Facult\'e de Sciences de Tunis, Universit\'e Tunis-El Manar, Tunisia.}
 \email{k.mehrez@yahoo.fr}
\address{Sergei M. Sitnik, Voronezh Institute of the Russian Ministry of Internal Affairs, Voronezh, Russia.}
\email{pochtaname@gmail.com}
\begin{abstract}
 In this paper, our aim is  to show some mean value inequalities for
the Fox-Wright functions, such as Tur\'an--type inequalities, Lazarevi\'c and Wilker--type inequalities. As applications we derive some new type inequalities for hypergeometric functions and the four--parametric Mittag--Leffler functions.  Furthermore, we prove monotonicity of ratios for sections of series of Fox-Wright functions, the results is also closely connected with Tur\'an--type inequalities. Moreover, some other type inequalities are also presented. At the end of the paper, some problems stated, which may be of interest for further research.
\end{abstract}
\maketitle

\section{Introduction}
In a series of recent papers,  the authors have studied certain functional inequalities and geometric properties for a some  special functions, for example, the classical
Gauss and Kummer hypergeometric functions, as well the generalized hypergeometric
functions \cite{KhSi3}, classical and generalized Mittag-Leffler functions \cite{KhSi1, KhSi2} and the Wright function \cite{Khaled1}.  Here, in our present investigation,  we generalize some these results to the Fox-Wright function ${}_p\Psi_q.$ 

Here, and in what follows, we use  ${}_p\Psi_q$ to denote the Fox-Wright  generalization of the familiar hypergeometric  ${}_p F_q$  function with $p$ numerator and $q$ denominator parameters (see \cite{M}), defined by (cf., e.g.,\cite[p. 4, Eq. (2.4)]{SS},) 
\begin{equation}\label{11}
{}_p\Psi_q\Big[_{(\beta_1,B_1),...,(\beta_q,B_q)}^{(\alpha_1,A_1),...,(\alpha_p,A_p)}\Big|z \Big]={}_p\Psi_q\Big[_{(\beta_q,B_q)}^{(\alpha_p,A_p)}\Big|z \Big]=\sum_{k=0}^\infty\frac{\prod_{l=1}^p\Gamma(\alpha_l+kA_l)}{\prod_{j=1}^q\Gamma(\beta_l+kB_l)}\frac{z^k}{k!},
\end{equation}
where $A_l\geq0,\; l=1,...,p; B_j\geq0,$ and $l=1,...,q.$   The series (\ref{11}) converges absolutely and uniformly for all bounded $|z|,\;z\in\mathbb{C}$ when
$$\epsilon=1+\sum_{l=1}^qB_l-\sum_{l=1}^pA_l>0.$$ 
The generalized hypergeometric function ${}_p F_q$ id defined by
\begin{equation}\label{12}
{}_p F_q\left[^{\alpha_1,...,\alpha_p}_{\beta_1,...,\beta_q}\Big|z\right]=\sum_{k=0}^\infty\frac{\prod_{l=1}^p(\alpha_l)_k}{\prod_{l=1}^q(\beta_l)_k}\frac{z^k}{k!}
\end{equation}
where, as usual, we make use of the following notation:
$$(\tau)_0=1, \textrm{and}\;\;(\tau)_k=\tau(\tau+1)...(\tau+k-1)=\frac{\Gamma(\tau+k)}{\Gamma(\tau)},\;k\in\mathbb{N},$$
to denote the shifted factorial or the Pochhammer symbol. Obviously, we find from the definitions (\ref{11}) and (\ref{12}) that 
\begin{equation}\label{13}
{}_p\Psi_q\Big[_{(\beta_1,1),...,(\beta_q,1)}^{(\alpha_1,1),...,(\alpha_p,1)}\Big|z \Big]=\frac{\Gamma(\alpha_1)...\Gamma(\alpha_p)}{\Gamma(\beta_1)...\Gamma(\beta_q)}{}_p F_q\left[^{\alpha_1,...,\alpha_p}_{\beta_1,...,\beta_q}\Big|z\right]
\end{equation}
We define the normalized Fox-Wright function  ${}_p\Psi_q^*$ by  
\begin{equation}\label{12}
{}_p\Psi_q^*\Big[_{(\beta_1,B_1),...,(\beta_q,B_q)}^{(\alpha_1,A_1),...,(\alpha_p,A_p)}\Big|z \Big]=\frac{\prod_{i=1}^q\Gamma(\beta_i)}{\prod_{i=1}^p\Gamma(\alpha_i}\sum_{k=0}^\infty\frac{\prod_{l=1}^p\Gamma(\alpha_l+kA_l)}{\prod_{l=1}^q\Gamma(\beta_l+kB_l)}\frac{z^k}{k!}.
\end{equation}

The Mittag--Leffler functions with $2n$  parameters are defined for $B_j\in\mathbb{R}\;\;(B_1^2+...+B_n^2\neq0)$ and $\beta_j\in\mathbb{C}\;\;(j=1,...,n\in\mathbb{N}),$ by the series
\begin{equation}\label{197}
E_{(B,\beta)_n}(z)=\sum_{k=0}^\infty\frac{z^k}{\prod_{j=1}^n\Gamma(\beta_j+kB_j)},\;z\in\mathbb{C}.
\end{equation}
When $n = 1$, the definition in (\ref{197})  coincides with the definition of the two--parametric Mittag--Leffler function
\begin{equation}\label{198}
E_{(B,\beta)_1}(z)=E_{B, \beta}(z)=\sum_{k=0}^\infty\frac{z^k}{\Gamma(\beta+kB)},\;z\in\mathbb{C},
\end{equation}
and and similarly for $n = 2$, where $E_{(B, \beta)_2}(z)$  coincides with the four--parametric Mittag--Leffler function
\begin{equation}\label{199}
E_{(B,\beta)_2}(z)=E_{B_1, \beta_1; B_2,\beta_2}(z)=\sum_{k=0}^\infty\frac{z^k}{\Gamma(\beta_1+kB_1)\Gamma(\beta_2+kB_2)},\;z\in\mathbb{C},
\end{equation}
is closer by its properties to the Wright function $W_{B,\beta}(z)$ defined by
\begin{equation}
W_{B,\beta}(z)=\sum_{k=0}^\infty\frac{z^k}{k!\Gamma(\beta_1+kB_1))},\;z\in\mathbb{C}.
\end{equation}
The generalized $2n-$parametric Mittag-Leffler function $E_{(\beta,B)_n}(z)$ can be represented in terms of the  Fox--Wright function ${}_p\Psi_q(z)$ by
\begin{equation}\label{199}
E_{(B,\beta)_n}(z)=E_{B_1,\beta_1;...;B_n,\beta_n}(z)={}_1\Psi_n\Big[_{(\beta_1,B_1),...,(\beta_n,B_n)}^{\;\;\;\;\;\;\;\;(1,1)\;\;\;\;}\Big|z \Big],\;z\in\mathbb{C}.
\end{equation}
\\
Throughout this paper, we adopt the following convention:
$$\alpha=(\alpha_1,...,\alpha_p),\;\beta=(\beta_1,...,\beta_q),\;A=(A_1,...,A_p),\;B=(B_1,...,B_q)$$ 
and
$${}_p\Psi_q\Big[_{(\beta_q,B_q)}^{(\alpha_p,A_p)}\Big|z \Big]={}_p\Psi_q\Big[_{(\beta_1, B_1),(\beta_{q-1},B_{q-1})}^{\;\;\;\;(\alpha_p,A_p)}\Big|z \Big]={}_p\Psi_q\Big[_{\;\;\;\;(\beta_q,B_q)}^{(\alpha_1, A_1),(\alpha_{p-1},A_{p-1})}\Big|z \Big].$$

\vspace{0,7cm}
The present sequel to some of the aforementioned investigations is organized as follows.
In Section 2, we state some useful lemmas which will be needed in the proofs of our results. In section 3, we present some Tur\'an type inequalities for the Fox--Wright functions ${}_p\Psi_q(z).$ As a consequence, we deduce the Tur\'an type inequalities for the hypergeometric functions ${}_pF_q(z)$ and for the  $2n-$parametric  Mittag--Leffler functions $E_{(B,\beta)_n}(z).$  Moreover, we prove monotonicity of ratios for sections of series of the Fox--Wright functions, the result is also closely connected with Tur\'an--type inequalities. In section 4, we give the Lazarevi\'c and Wilker type inequalities for the Fox--Wright function ${}_1\Psi_2(z).$ As applications, we derive the Lazarevi\'c and Wilker type inequalities for the for the hypergeometric functions ${}_1F_{2}(z)$ and for the four--parametric Mittag--Leffler functions $E_{B_1,\beta_1; 1\beta_2}(z).$ In section 5,  we present some other inequalities for the Fox--Wright function ${}_p\Psi_{p+1}(z).$  Finally, in Section 6, we pose two open problems, which may be interest for further research.\\

Each of the following definitions will be used in our investigation.\\

\begin{define} A function $f:[a,b]\subseteq\mathbb{R}\rightarrow\mathbb{R}$  is said to be log-convex if its natural logarithm $\log f$ is convex, that is, for all $x,y\in[a,b]$ and $\alpha\in[0,1]$ we have
$$f(\alpha x+(1-\alpha)y)\leq[f(x)]^\alpha[f(y)]^{1-\alpha}.$$
If the above inequality is reversed then $f$ is called a log-concave function. It is also known that if $g$ is differentiable, then $f$ is log-convex (log-concave) if and only if $f^\prime/f$ is increasing (decreasing). 
\end{define}

\section{Preliminary Lemmas}

In the proof of the main result we will need the following  lemmas.

\begin{lemma}\label{l1}
Let $(a_{n})$ and $(b_{n})$ $(n=0,1,2...)$ be real numbers, such that $b_{n}>0,\;n=0,1,2,...$ and $\left(\frac{a_{n}}{b_{n}}\right)_{n\geq 0}$ is increasing (decreasing), then $\left(\frac{a_{0+}...+a_{n}}{b_{0}+...+b_{n}}\right)_{n}$ is also increasing (decreasing).
\end{lemma}

The second lemma is about the monotonicity of two power series, see \cite{PV} for more details.

\begin{lemma}\label{l2}
Let $(a_{n})$ and $(b_{n})$ $(n=0,1,2...)$ be real numbers and
let the power series $A(x)=\sum_{n=0}^{\infty}a_{n}x^{n}$ and $B(x)=\sum_{n=0}^{\infty}b_{n}x^{n}$
be convergent for $|x|<r$. If $b_{n}>0,\, n=0,1,2,...$ and  the
sequence $\left(\frac{a_{n}}{b_{n}}\right)_{n\geq0}$is (strictly)
increasing (decreasing) , then the function $\frac{A(x)}{B(x)}$ is also
(strictly) increasing on $[0,r)$.
\end{lemma}


\section{Tur\'an type inequalities for Fox-Wright function}

Our first main result is asserted by the following theorem.

\begin{theorem}\label{T01}Let $\alpha,\beta>0,$ and $A, B\geq0$ such that $\epsilon>0.$ Then the Fox-Wright function ${}_p\Psi_q$  possesses the following Tur\'an type inequality:
\begin{equation}\label{001}
{}_p\Psi_q\Big[^{(\alpha_1, A_1),(\alpha_{p-1}, A_{p-1})}_{\;\;\;\;\;\;\;\;(\beta_q,B_q)}\Big| z\Big]{}_p\Psi_q\Big[^{(\alpha_1+2, A_1),(\alpha_{p-1}, A_{p-1})}_{\;\;\;\;\;\;\;\;\;\;\;(\beta_q, B_q)}\Big|z\Big]-\Bigg({}_p\Psi_q\Big[^{(\alpha_1+1, A_1),(\alpha_{p-1}, A_{p-1})}_{\;\;\;\;\;\;\;\;\;\;\;(\beta_q, B_q)}\Big| z\Big]\Bigg)^2>0,\;\;\Big(z\in(0,\infty)\Big).
\end{equation}
\end{theorem}
\begin{proof}By using the Cauchy product formula, we have
$$\Big({}_p\Psi_q\Big[^{(\alpha_1+1, A_1),(\alpha_{p-1}, A_{p-1})}_{\;\;\;\;\;\;\;\;\;(\beta_q, B_q)}\Big|z\Big]\Big)^2=\sum_{k=0}^\infty\sum_{j=0}^k\frac{\Gamma(\alpha_1+jA_1+1)\Gamma(\alpha_1+(k-j)A_1+1)\prod_{i=2}^{p}\Gamma(\alpha_i+jA_i)\Gamma(\alpha_i+(k-j)A_i)z^k}{j!(k-j)!\bigg[\prod_{i=1}^q\Gamma(\beta_i+jB_i)\Gamma(\beta_i+(k-j)B_i)\bigg]},$$
and
$${}_p\Psi_q\Big[^{(\alpha_1,A_1),(\alpha_{p-1},A_{p-1})}_{\;\;\;\;(\beta_q,B_q)}\Big|z\Big]{}_p\Psi_q\Big[^{(\alpha_1+2, A_1),(\alpha_{p-1}, Av)}_{\;\;\;\;\;\;(\beta_q, B_q)}\Big| z\Big]=$$
$$=\sum_{k=0}^\infty\sum_{j=0}^k\frac{\Gamma(\alpha_1+jA_1)\Gamma(\alpha_1+(k-j)A_1+2)\prod_{i=2}^{p}\Gamma(\alpha_i+jA_i)\Gamma(\alpha_i+(k-j)A_i)z^k}{j!(k-j)!\bigg[\prod_{i=1}^q\Gamma(\beta_i+jB_i)\Gamma(\beta_i+(k-j)B_i)\bigg]}.$$
Thus
\begin{equation*}
\Big({}_p\Psi_q\Big[^{(\alpha_1+1, A_1),(\alpha_{p-1}, A_{p-1})}_{\;\;\;\;(\beta_q, B_q)}\Big|z\Big]\Big)^2-{}_p\Psi_q\Big[^{(\alpha_1, A_1),(\alpha_{p-1}, A_{p-1})}_{\;\;\;\;(\beta_1,B_1)}\Big|z\Big]{}_p\Psi_q\Big[^{(\alpha_1+2,A_1),(\alpha_{p-1}, A_{p-1})}_{\;\;\;\;(\beta_q, B_q)}\Big|z\Big]=\sum_{k=0}^\infty\sum_{j=0}^k K_{j,k}^{(1)}T_{j,k}^{(1)}(\alpha_1, A_1)z^k,
\end{equation*}
where $T_{j,k}^{(1)}(\alpha_1, A_1)$ and $K_{j,k}^{(1)}$  are defined by 
\begin{equation*}
\begin{split}T_{j,k}^{(1)}(\alpha_1, A_1)&=\Gamma(\alpha_1+jA_1+1)\Gamma(\alpha_1+(k-j)A_1+1)-\Gamma(\alpha_1+jA_1)\Gamma(\alpha_1+(k-j)A_1+2)\\
&=[(2j-k)-1]\Gamma(\alpha_1+jA_1)\Gamma(\alpha_1+(k-j)A_1+1),
\end{split}
\end{equation*}
and 
$$K_{j,k}^{(1)}=\frac{\prod_{i=2}^{p}\Gamma(\alpha_i+jA_i)\Gamma(\alpha_i+(k-j)A_i)}{j!(k-j)!\bigg[\prod_{i=1}^q\Gamma(\beta_i+jB_i)\Gamma(\beta_i+(k-j)B_i)\bigg]}.$$
\noindent \textbf{Case 1.} Let $n$ be an even positive integer. Then
\begin{equation}
\begin{split}
\sum_{j=0}^k K_{j,k}^{(1)}T_{j,k}^{(1)}(\alpha_1, A_1)&=\sum_{j=0}^{\frac{k}{2}-1}K_{j,k}^{(1)}T_{j,k}^{(1)}(\alpha_1, A_1)+\sum_{j=\frac{k}{2}+1}^{k}K_{j,k}^{(1)}T_{j,k}^{(1)}(\alpha_1, A_1)\\&+K_{k/2,k}^{(1)}T_{k/2,k}^{(1)}(\alpha_1, A_1)\\
&=\sum_{j=0}^{k/2-1} K_{j,k}^{(1)}(T_{j,k}^{(1)}(\alpha_1, A_1)+T_{k-j,k}^{(1)}(\alpha_1, A_1))\\&+K_{k/2,k}^{(1)}T_{k/2,k}^{(1)}(\alpha_1, A_1)\\
&=\sum_{j=0}^{[\frac{k-1}{2}]} K_{j,k}^{(1)}(T_{j,k}^{(1)}(\alpha_1, A_1)+T_{k-j,k}^{(1)}(\alpha_1, A_1))\\&+K_{k/2,k}^{(1)}T_{k/2,k}^{(1)}(\alpha_1, A_1),
\end{split}
\end{equation}
where, as usual, $[k]$ denotes the greatest integer part of $k\in\mathbb{R}.$\\
\noindent \textbf{Case 2.} Let $n$ be an odd positive integer. Then, just as in Case 1, we get
\begin{equation*}
\begin{split}
\sum_{j=0}^k K_{j,k}^{(1)}T_{j,k}^{(1)}(\alpha_1, A_1)&=\sum_{j=0}^{[\frac{k-1}{2}]} K_{j,k}^{(1)}(T_{j,k}^{(1)}(\alpha_1, A_1)+T_{k-j,k}^{(1)}(\alpha_1, A_1))\\&+K_{k/2,k}^{(1)}T_{k/2,k}^{(1)}(\alpha_1, A_1).
\end{split}
\end{equation*}
Thus, by combining Case 1 and Case 2, we have
$$\Big({}_p\Psi_q\Big[^{(\alpha_1+1, A_1),(\alpha_{p-1}, A_{p-1})}_{\;\;\;\;\;\;\;\;(\beta_q, B_q)}\Big| z\Big]\Big)^2-{}_p\Psi_q\Big[^{(\alpha_1 ,A_1),(\alpha_{p-1}, A_{p-1})}_{\;\;\;\;\;\;\;\;(\beta_q, B_q)}\Big|z\Big]{}_p\Psi_q\Big[^{(\alpha_1+2, A_1),(\alpha_{p-1}, A_{p-1})}_{\;\;\;\;\;\;\;\;(\beta_q, B_q)}\Big|z\Big]=$$
\begin{equation}
\begin{split}
\;\;\;\;\;&=\sum_{k=0}^\infty\sum_{j=0}^{[\frac{k-1}{2}]} K_{j,k}^{(1)}\Big(T_{j,k}^{(1)}(\alpha_1, A_1)+T_{k-j,k}^{(1)}(\alpha_1, A_1)\Big)+K_{k/2,k}^{(1)}T_{k/2,k}^{(1)}(\alpha_1, A_1)z^k,
\end{split}
\end{equation}
which, upon simplifying, yields
\begin{equation*}
\begin{split}
T_{j,k}^{(1)}(\alpha_1, A_1)+T_{k-j,k}^{(1)}(\alpha_1, A_1))&=-\Big[(2k-j)^2+(2\alpha_1+kA_1)\Big]\Gamma(\alpha_1+(k-j)A_1)\Gamma(\alpha_1+jA_1)<0.
\end{split}
\end{equation*}
On the other hand, we have
$$T_{k/2,k}^{(1)}(\alpha_1, A_1)=-\left(\alpha_1+\frac{k}{2}\right)\Gamma^2\Big(\alpha_1+\frac{k}{2}A_1\Big)<0,$$
which evidently completes the proof of Theorem \ref{T01}.
\end{proof}

Letting in (\ref{001}) the values $A=B=1$  and using the formula (\ref{13}),  we the following Tur\'an type inequality for the hypergeometric function ${}_pF_q.$

\begin{coro}Let $\alpha,\;\beta>0.$ Then the following Tur\'an type inequality:
\begin{equation}\label{003}
{}_pF_q\left[^{\alpha_1,\alpha_2,..., \alpha_p}_{\;\;\;\;\beta_1,...,\beta_q}\Big|z\right]{}_pF_q\left[^{\alpha_1+2,\alpha_2,...,\alpha_p}_{\;\;\;\;\beta_1,...,\beta_q}\Big|z\right]-\frac{\alpha_1}{\alpha_1+1}\Big({}_pF_q\left[^{\alpha_1+1,\alpha_2,...,\alpha_q}_{\;\;\;\;\beta_1,...,\beta_q}\Big|z\right]\Big)^2>0,
\end{equation}
holds true for all $z\in(0,\infty).$
\end{coro}

\begin{theorem}\label{T1}Let $\alpha,\beta>0,$ and $A, B\geq0$ such that $\epsilon>0.$ Then the following Tur\'an type inequalities
\begin{equation}\label{1}
{}_p\Psi_q\Big[^{\;\;\;\;(\alpha_p,A_p)}_{(\beta_1,B_1),(\beta_{q-1},B_{q-1})}\Big|z\Big]{}_p\Psi_q\Big[^{\;\;\;\;(\alpha_p,A_p)}_{(\beta_1+2,B_1),(\beta_{q-1},B_{q-1})}\Big|z\Big]-\frac{\beta_1}{\beta_1+1}\Bigg({}_p\Psi_q\Big[^{\;\;\;\;(\alpha_p,A_p)}_{(\beta_1+1,B_1),(\beta_{q-1},B_{q-1})}\Big|z\Big]\Bigg)^2\geq0,
\end{equation}
holds true for all $z\in(0,\infty).$ Moreover, the Hypergeometric function $_pF_q$ satisfies the following Tur\'an type inequality
\begin{equation}\label{2}
{}_pF_q\left[^{\;\;\;\;\alpha_1,..., \alpha_p}_{\beta_1,\beta_2,...,\beta_q}\Big|z\right] {}_pF_q\left[^{\;\;\;\;\alpha_1,..., \alpha_p}_{\beta_1+2,\beta_2,...,\beta_q}\Big|z\right]-\Big({}_pF_q\left[^{\;\;\;\;\alpha_1,..., \alpha_p}_{\beta_1+1, \beta_2,...,\beta_q}\Big|z\right]\Big)^2\geq0,\;\Big(\;z\in(0,\infty)\;\Big).
\end{equation}
\end{theorem}
\begin{proof} We set 
\begin{equation}
\tilde{{}_p\Psi_q}\Big[^{(\alpha_p, A_p)}_{(\beta_q, B_q)}\Big| z\Big]=\Gamma(\beta_1){}_p\Psi_q\Big[^{(\alpha_p, A_p)}_{(\beta_q, B_q)}\Big| z\Big].
\end{equation}
By using the Cauchy product we get 
\begin{equation}
\tilde{{}_p\Psi_q}\Big[^{\;\;\;\;(\alpha_p,  A_p)}_{(\beta_1, B_1),(\beta_{q-1}, B_{q-1})}\Big| z\Big]\tilde{{}_p\Psi_q}\Big[^{\;\;\;\;(\alpha_p, A_p)}_{(\beta_1+2, B_1),(\beta_{q-1}, B_{q-1})}\Big|z\Big]-\tilde{{}_p\Psi_q}^2\Big[^{\;\;\;\;(\alpha_p, A_p)}_{(\beta_1+1, B_1),(\beta_{q-1}, B_{q-1})}\Big| z\Big]=
\end{equation}
$$\;\;\;\;\;\;\;\;\;\;\;\;\;\;\;\;\;\;\;\;\;\;\;\;\;\;\;\;\;\;\;\;\;\;\;\;\;\;\;\;\;\;\;\;\;\;\;\;\;\;\;\;\;\;\;\;\;\;\;\;\;\;\;\;\;\;\;\;\;\;\;\;\;\;\;=\Gamma(\beta_1)\Gamma(\beta_1+1)\sum_{k=0}^\infty\sum_{j=0}^k K_{k,j}^{(2)}T_{k,j}^{(2)}(\beta_1,B_1)z^k,$$
where
\begin{equation}
K_{k,j}^{(2)}=\frac{\prod_{i=1}^p\Gamma(\alpha_i+jA_i)\Gamma(\alpha_i+(k-j)A_i)}{j!(k-j)!\prod_{i=2}^q\Gamma(\beta_i+jB_i)\Gamma(\beta_i+(k-j)B_i)}
\end{equation}
and
\begin{equation}
T_{k,j}^{(2)}(\beta_1,B_1)=\frac{\beta_1B_1(2j-k)+jB_1}{\Gamma(\beta_1+jB_1+1)\Gamma(\beta_1+(k-j)B_1+2)}.
\end{equation}
If $k$ is even, we have 
\begin{equation}
\begin{split}
\sum_{j=0}^k K_{k,j}^{(2)}T_{k,j}^{(2)}(\beta_1,B_1)&=\sum_{j=0}^{k/2-1} K_{k,j}^{(2)}T_{k,j}^{(2)}(\beta_1,B_1)+\sum_{j=k/2+1}^k K_{k,j}^{(2)}T_{k,j}^{(2)}(\beta_1,B_1)\\&+K_{k,k/2}^{(2)}T_{k,k/2}^{(2)}(\beta_1,B_1)\\
&=\sum_{j=0}^{k/2-1} K_{k,j}^{(2)}T_{k,j}^{(2)}(\beta_1,B_1)+\sum_{j=0}^{k/2-1} K_{k,j}^{(2)}T_{k,k-j}^{(2)}(\beta_1,B_1)\\&+K_{k,k/2}^{(2)}T_{k,k/2}^{(2)}(\beta_1,B_1)\\
&=\sum_{j=0}^{[(k-1)/2]} K_{k,j}^{(2)}\left(T_{k,j}^{(2)}(\beta_1,B_1)+T_{k,k-j}^{(2)}(\beta_1,B_1)\right)+K_{k,k/2}^{(2)}T_{k,k/2}^{(2)}(\beta_1,B_1).
\end{split}
\end{equation}
where $[.]$ denotes the greatest integer function.  Similarly, if $k$ is odd, then
$$\sum_{j=0}^k K_{k,j}^{(2)}T_{k,j}^{(2)}(\beta_1,B_1)=\sum_{j=0}^{[(k-1)/2]} K_{k,j}^{(2)}(\beta_1,B_1)\left(T_{k,j}^{(2)}(\beta_1,B_1)+T_{k,k-j}^{(2)}(\beta_1,B_1)\right)+K_{k,k/2}^{(2)}T_{k,k/2}^{(2)}(\beta_1,B_1).$$
A simple computation we get 
\begin{equation}
T_{k,j}^{(2)}(\beta_1,B_1)+T_{k,k-j}^{(2)}(\beta_1,B_1)=\frac{B_1\beta_1(k-2j)^2+j^2B_1^2+B_1^2(k-j)^2+k(B_1+\beta_1)}{\Gamma(\beta_1+jB_1+2)\Gamma(\beta_1+(k-j)B_1+2)}\geq0,
\end{equation}
and using the fact 
\begin{equation}
K_{k,k/2}^{(2)}T_{k,k/2}^{(2)}(\beta_1,B_1)=\frac{B_1k\prod_{i=1}^p\Gamma^2(\alpha_i+\frac{kA_i}{2})}{2\Gamma^2(\frac{k}{2}+1)\Gamma(\beta_1+\frac{kB_1}{2}+1)\Gamma(\beta_1+\frac{kB_1}{2}+2)\prod_{i=2}^q\Gamma^2(\beta_i+\frac{kB_i}{2})}\geq0,
\end{equation}
we deduce that 
\begin{equation}
\tilde{{}_p\Psi_q}\Big[^{\;\;\;\;\;\;\;(\alpha_p, A_p)}_{(\beta_1, B_1),(\beta_{q-1}, B_{q-1})}\Big|z\Big]\tilde{{}_p\Psi_q}\Big[^{\;\;\;\;\;\;\;(\alpha_p, A_p)}_{(\beta_1+2,B_1),(\beta_{q-1}, B_{q-1})}\Big| z\Big]-\tilde{{}_p\Psi_q}^2\Big[^{\;\;\;\;\;\;\;(\alpha_p, A_p)}_{(\beta_1+1,B_1),(\beta_{q-1}, B_{q-1})}\Big|z\Big]\geq0.
\end{equation}
It is important to mention here that there is another proof of the inequalities (\ref{1}). Namely, we consider the expression
$$\tilde{{}_p\Psi_q}\Big[^{(\alpha_p, A_p)}_{(\beta_q, B_q)}\Big| z\Big]=\sum_{n=0}^\infty \delta_{A,B,n}(\alpha,\beta)z^n,\;\textrm{where}\;\; \delta_{A,B,n}(\alpha,\beta)=\frac{\Gamma(\beta_1)\prod_{i=1}^p\Gamma(\alpha_i+nA_i)}{\Gamma(\beta_1+nB_1)\prod_{i=2}^q\Gamma(\beta_i+nB_i)}.$$
Computations show that for each $n\geq0$ we get
$$\frac{\partial^2\log[\delta_{A,B,n}(\alpha,\beta)] }{\partial \beta_1^2}=\psi^\prime(\beta_1)-\psi(\beta_1+nB_1),$$
where $\psi(x)=\frac{\Gamma^\prime(x)}{\Gamma(x)}$ is the digamma function. It is well known that the function $x\mapsto \psi(x)$ is concave on $(0,\infty),$ i.e. the trigamma function $x\mapsto\psi^\prime(x)$ is decreasing on $(0,\infty).$ Therefore, the function $\beta_1\mapsto\delta_{A,B,n}(\alpha,\beta)$ is log-convex on $(0,\infty).$ Thus, the function $\beta_1\mapsto\tilde{{}_p\psi_q}\Big[^{(\alpha_p,A_q)}_{(\beta_q,B_q)}; z\Big]$ is also log-convex on $(0,\infty).$ So, for all $\alpha, \beta, \beta_1^\prime>0,$ and $t\in[0,1],$ we get
\begin{equation}\label{a}
\tilde{{}_p\psi_q}\Big[^{\;\;\;\;\;\;\;\;\;\;\;\;(\alpha_p, A_p)}_{(t\beta_1+(1-t)\beta_1^{\prime}, B_1),(\beta_{q-1}, B_{q-1})}\Big| z\Big]\leq\Big(\tilde{{}_p\psi_q}\Big[^{\;\;\;\;\;\;(\alpha_p, A_p)}_{(\beta_1,B_1),(\beta_{q-1}, B_{q-1})}\Big| z\Big]\Big)^t\Big(\tilde{{}_p\psi_q}\Big[^{\;\;\;\;\;\;(\alpha_p, A_p)}_{(\beta_1^{\prime}, B_1),(\beta_{q-1}, B_{q-1})}\Big|z\Big]\Big)^{1-t}.
\end{equation}
Letting $t=1/2$ and $\beta_1^{\prime}=\beta_1+2,$ in the above inequality we deduce that the inequality (\ref{1}) holds true.
The inequality (\ref{2}) follow by using the inequalities (\ref{1}) and (\ref{13}). So, the proof of Theorem \ref{T1} is completes.
\end{proof}

Choosing in (\ref{1}) the values $p=1,\alpha_1=A_1=1,$  we obtain the following Tur\'an type inequality for the The generalized $2n-$parametric Mittag-Leffler function:  

\begin{coro} Let $\beta>0$ and $B\geq0.$  Then the following Tur\'an type inequality
\begin{equation}\label{yyy}
E_{B_1,\beta_1;...;B_n,\beta_n}(z)E_{B_1, \beta_1+2;...;B_n,\beta_n}(z)-\frac{\beta_1}{\beta_1+1}\Big(E_{B_1, \beta_1+1;...;B_n,\beta_n}(z)\Big)^2\geq0,
\end{equation}
holds true for all $z>0.$
\end{coro}
\begin{coro} The generalized hypergeometric function $_2F_2$ possesses the following inequality:
\begin{equation}{}_2F_2\left[^{\beta_1-\alpha_1-1,\;\;f+1}_{\beta_1,\;\;f}\Big|z\right] {}_2F_2\left[^{\beta_1-\alpha_1+1,\;\; g+1}_{\beta_1+2,\;\; g}\Big|z\right]-\Big({}_2F_2\left[^{\beta_1-\alpha_1,\;\; h+1}_{\beta_1+1,\;\; h}\Big|z\right]\Big)^2\geq0,\;\Big(\;z\in(-\infty,0)\Big)
\end{equation}
with
$$f=\frac{\beta_2(1+\alpha_1-\beta_1)}{\alpha_1-\beta_2},\;g=\frac{\beta_2(\alpha_1-\beta_1-1)}{\alpha_1-\beta_2}\;\textrm{and}\;\;h=\frac{\beta_2(\alpha_1-\beta_1)}{\alpha_1-\beta_2}$$
\end{coro}
\begin{proof}	By means of the Kummer transformation for the hypergeometric function $_2F_2$ reported by Paris \cite[Eq. 4]{Paris}
$$_2F_2\left[_{b,\;\;c}^{a,\;\;c+1}\Big|z\right]=e^z\; _2F_2\left[_{b,\;\;f_1}^{a,\;\;f_1+1}\Big|-z\right],\;\textrm{with}\;\;f_1=\frac{c(1+a-b)}{a-c},$$
and the Tur\'an type inequality (\ref{2}) lead to the asserted inequality.
\end{proof}
\begin{remark}
\noindent \textbf{a.} If we choose $p=q=1,B_1=\alpha,\beta_1=\beta$ and $A_1=0$  in (\ref{1}) we obtain the following Tur\'an type inequalities for the Wright function \cite[Theorem 3.1]{Khaled1}:
\begin{equation*}
\mathcal{W}_{\alpha,\beta}(z)\mathcal{W}_{\alpha,\beta+2}(z)-\mathcal{W}_{\alpha,\beta+1}^2(z)\geq0,
\end{equation*}
where $\mathcal{W}_{\alpha,\beta}(z)=\Gamma(\beta)W_{\alpha,\beta}(z).$\\
\noindent \textbf{b.} Letting $n=2$ in (\ref{yyy}), we deduce the following Tur\'an type inequalities for the Mittag--Leffler function  \cite[Theorem 1]{KhSi1}:
\begin{equation*}
\mathbb{E}_{\alpha,\beta}(z)\mathbb{E}_{\alpha,\beta+2}(z)-\mathbb{E}_{\alpha,\beta+1}^2(z)\geq0,
\end{equation*}
where $\mathbb{E}_{\alpha,\beta}(z)=\Gamma(\beta)E_{\alpha,\beta}(z).$
\end{remark}

\begin{theorem}\label{T2} Let $\alpha,\beta,\beta_1^\prime>0,$ and $A, B\geq0$ such that $\epsilon>0.$ . If $\beta_1^\prime<\beta_1,\;(\beta_1<\beta_1^\prime),$ then the function 
$$z\mapsto {}_p\Psi_q\Big[_{(\beta_1, B_1),(\beta_{q-1}, B_{q-1})}^{\;\;\;\;\;\;(\alpha_p, A_p)}\Big|z \Big]\Big/{}_p\Psi_q\Big[_{(\beta_1^\prime, B_1),(\beta_{q-1}, B_{q-1})}^{\;\;\;\;\;\;(\alpha_p, A_p)}\Big|z \Big],$$
is decreasing (increasing) on $(0,\infty).$ Moreover, the following inequality 
\begin{equation}
{}_p\Psi_q\Big[_{(\beta_1+B_1, B_1),(\beta_{q-1}+B_{q-1}, B_{q-1})}^{\;\;\;\;\;\;\;(\alpha_p+A_p, A_p)}\Big|z \Big]{}_p\Psi_q\Big[_{(\beta_1^\prime, B_1),(\beta_{q-1}, B_{q-1})}^{\;\;\;\;\;(\alpha_p, A_p)}\Big|z \Big]\leq(\geq){}_p\Psi_q\Big[_{(\beta_1^\prime+B_1, B_1),(\beta_{q-1}+B_{q-1}, B_{q-1})}^{\;\;\;\;\;(\alpha_p+A_p, A_p)}\Big|z \Big]{}_p\Psi_q\Big[_{(\beta_q,B_q)}^{(\alpha_p, A_p)}\Big|z \Big], 
\end{equation}
holds. 
\end{theorem}
\begin{proof}Let
 $$
\frac{{}_p\Psi_q\Big[_{(\beta_q, B_q)}^{(\alpha_p, A_p)}\Big|z \Big]}{{}_p\Psi_q\Big[_{(\beta_q^\prime, B_q)}^{(\alpha_p, A_p)}\Big|z \Big]}=\sum_{k=0}^\infty U_k^{0}(\alpha,A;\beta,B)z^k\Big/\sum_{k=0}^\infty V_k^{0}(\alpha,A;\beta^\prime,B)z^k,$$
where 
$$ U_k^{0}(\alpha,A;\beta,B)=\frac{\prod_{i=1}^p\Gamma(\alpha_i+kA_i)}{\Gamma(\beta_1+kB_1)\prod_{i=2}^q\Gamma(\beta_i+kB_i)},\;
\textrm{and}\;V_k^{0}(\alpha,A;\beta^\prime,B)=\frac{\prod_{i=1}^p\Gamma(\alpha_i+kA_i)}{\Gamma(\beta_1^\prime+kB_1)\prod_{i=2}^q\Gamma(\beta_i+kB_i)}.$$
We set 
$$W_k^{0}=\frac{U_k^{0}(\alpha,A;\beta,B)}{V_k^{0}(\alpha,A;\beta^\prime,B)}=\frac{\Gamma(\beta_1^\prime+kB_1)}{\Gamma(\beta_1+kB_1)}.$$
Using the fact that  the Gamma function $\Gamma(z)$ is log-convex on $(0,\infty)$, we deduce that the ratios $z\mapsto\frac{\Gamma(z+a)}{\Gamma(z)}$ is increasing on $(0,\infty)$ when $a>0.$ Thus implies that the following inequality
\begin{equation}\label{32}
\frac{\Gamma(z+a)}{\Gamma(z)}\leq\frac{\Gamma(z+a+b)}{\Gamma(z+b)}
\end{equation}
holds for all $a,b,z>0.$ In the case  $\beta_1^\prime<\beta_1$, we let $z=\beta_1^\prime+kB_1,\;a=B_1$ and $b=\beta_1-\beta_1^\prime>0$ in (\ref{32}) we obtain that
\begin{equation}
\frac{W_{k+1}^{0}}{W_k^{0}}=\frac{\Gamma(\beta_1^\prime+B_1+kB_1)\Gamma(\beta_1+kB_1)}{\Gamma(\beta_1^\prime+kB_1)\Gamma(\beta_1+B_1+kB_1)}\leq1.
\end{equation}
Thus, $W_{k+1}^{0}\leq W_k^{0}$ for all $k\geq0$ if and only if $\beta_1>\beta_1^\prime,$ and the function 
$$ z\mapsto {}_p\Psi_q\Big[_{(\beta_1, B_1),(\beta_{q-1}, B_{q-1})}^{\;\;\;\;\;\;(\alpha_p, A_p)}\Big|z \Big]\Big/{}_p\Psi_q\Big[_{(\beta_1^\prime, B_1),(\beta_{q-1}, B_{q-1})}^{\;\;\;\;\;\;(\alpha_p, A_p)}\Big|z \Big] $$ is decreasing on $(0,\infty)$ if $\beta_1>\beta_1^\prime,$ by means of Lemma \ref{l2}. In the case $\beta_1^\prime>\beta_1,$ we set $z=\beta_1+kB_1,\;a=B_1$ and $b=\beta_1^\prime-\beta_1>0$ in (\ref{32}), we conclude that $W_{k+1}^{0}\geq W_k^{0}$ for all $k\geq0.$ We thus implies that the function 
$$z\mapsto {}_p\Psi_q\Big[_{(\beta_1, B_1),(\beta_{q-1}, B_{q-1})}^{\;\;\;\;\;\;(\alpha_p, A_p)}\Big|z \Big]\Big/{}_p\Psi_q\Big[_{(\beta_1^\prime, B_1),(\beta_{q-1}, B_{q-1})}^{\;\;\;\;\;\;(\alpha_p, A_p)}\Big|z \Big]$$ is increasing on $(0,\infty)$ if $\beta_1^\prime>\beta_1,$ by Lemma \ref{l2}. Therefore,
 $$\left({}_p\Psi_q\Big[_{(\beta_1, B_1),(\beta_{q-1}, B_{q-1})}^{\;\;\;\;\;\;(\alpha_p, A_p)}\Big|z \Big]\Big/{}_p\Psi_q\Big[_{(\beta_1^\prime, B_1),(\beta_{q-1}, B_{q-1})}^{\;\;\;\;\;\;(\alpha_p, A_p)}\Big|z \Big]\right)^\prime\leq0,$$
if $\beta_1>\beta_1^\prime.$ Therefore, the differentiation formula
\begin{equation}\label{!!}
\left({}_p\Psi_q\Big[_{(\beta_q, B_q)}^{(\alpha_p, A_p)}\Big|z \Big]\right)^\prime={}_p\Psi_q\Big[_{(\beta_q+B_q, B_q)}^{(\alpha_p+A_p, A_p)}\Big|z \Big]
\end{equation}
 complete the proof of the asserted results immediately.
\end{proof}
\begin{remark} \noindent \textbf{a.} Letting in Theorem \ref{T2}, the values $A=B=1$, we conclude that, if $\beta_1^\prime<\beta_1$ (rep. $\beta_1<\beta_1^\prime$), then the function 
$$z\mapsto _pF_q\Big[_{\beta_1,...,\beta_q}^{\alpha_1,...,\alpha_p}\Big|z\Big]\Big/ {}_pF_q\Big[_{\beta_1^\prime,...,\beta_q}^{\alpha_1,...,\alpha_p}\Big|z\Big]$$
is decreasing (resp. increasing) on $(0,\infty).$ Consequently the following inequality hold true:
\begin{equation}
_pF_q\Big[_{\beta_1+1,...,\beta_q+1}^{\alpha_1+1,...,\alpha_p+1}\Big|z\Big] {}_pF_q\Big[_{\beta_1^\prime,...,\beta_q}^{\alpha_1,...,\alpha_p}\Big|z\Big]\leq \left(\frac{\beta_1}{\beta_1^\prime}\right) {}_pF_q\Big[_{\beta_1,...,\beta_q}^{\alpha_1,...,\alpha_p}\Big|z\Big] {}_pF_q\Big[_{\beta_1^\prime+1,...,\beta_q+1}^{\alpha_1+1,...,\alpha_p+1}\Big|z\Big]
\end{equation} 
when $\beta_1^\prime<\beta_1$ and $z>0.$ Moreover, the above inequality is reversed if $\beta_1<\beta_1^\prime$ and $z>0.$\\
\noindent \textbf{b.} Choosing $q=p+1, A_i=B_{i+1},\alpha_i=\beta_{i+1},i=1,...,p$ in Theorem \ref{T2}, we deduce that the ratios $z\mapsto W_{B_1,\beta_1}(z)/W_{B_1,\beta_1^\prime}(z)$ is decreasing (resp. increasing) on $(0,\infty)$ if $\beta_1^\prime<\beta_1$ (resp. $\beta_1<\beta_1^\prime.$) (cf. see \cite[Theorem 3.2]{Khaled1}), and consequently we obtain the following inequality  \cite[Theorem 3.2, eq. 3.2]{Khaled1}:
$$W_{B_1,\beta_1}(z)W_{B_1,\beta_1^\prime+B_1}(z)-W_{B_1,\beta_1^\prime}(z)W_{B_1,\beta_1+B_1}(z)\geq0,$$
when $\beta_1^\prime<\beta_1.$ The above inequality reduces to the following Tur\'an type inequality:
$$W_{1,2}^2(z)-W_{1,1}(z)W_{1,3}(z)\geq 0,\;\;(z>0) .$$
\noindent \textbf{c.} Choosing $p=\alpha_1=A_1=1$ and $q=1$ in Theorem \ref{T2}, we deduce that the ratios $z\mapsto E_{B_1,\beta_1}(z)/E_{B_1,\beta_1^\prime}(z)$ is decreasing (resp. increasing) on $(0,\infty)$ if $\beta_1^\prime<\beta_1$ (resp. $\beta_1<\beta_1^\prime.$) (cf. see \cite[Theorem 4]{KhSi1}), and we get 
\begin{equation}
E_{B_1,\beta_1}(z){}_1\Psi_1\Big[_{(\beta_1^\prime+B_1, B_1)}^{\;\;\;(2, 1)}\Big|z \Big]-E_{B_1,\beta_1^\prime}(z){}_1\Psi_1\Big[_{(\beta_1+B_1, B_1)}^{\;\;\;(2, 1)}\Big|z \Big]\geq 0,
\end{equation}
when $\beta_1^\prime<\beta_1.$  By using the familiar relationship:
$${}_1\Psi_1\Big[_{(\beta_1+B_1, B_1)}^{\;\;\;(2, 1)}\Big|z \Big]=(E_{B_1,\beta_1}(z))^\prime$$
and
$$\frac{d}{dz}E_{B_1,\beta_1}(z)=\frac{E_{B_1,\beta_1-1}(z)-(\beta_1-1)E_{B_1,\beta_1}(z)}{B_1z}$$
we obtain \cite[Theorem 4, eq. 10]{KhSi1}
$$E_{B_1,\beta_1}(z)E_{B_1,\beta_1^\prime-1}(z)-E_{B_1,\beta_1^\prime}(z)E_{B_1,\beta_1-1}(z)+(\beta_1-\beta_1^\prime)E_{B_1,\beta_1}(z)E_{B_1,\beta_1^\prime}(z)\geq0,\;(\;z>0).$$
\noindent \textbf{d.} A similar arguments to the proof of Theorem \ref{T2}, we obtain the following results: Let $\alpha,\beta,\alpha_1^\prime>0,$ and $A, B\geq0$ such that $\epsilon>0.$  If $\alpha_1<\alpha_1^\prime,\;(\textrm{resp.} \;\alpha_1^\prime<\alpha_1),$ then the function 
$$z\mapsto {}_p\Psi_q\Big[_{\;\;\;\;\;(\beta_q, B_q)}^{(\alpha_1, A_1),(\alpha_{p-1}, A_{p-1})}\Big|z \Big]\Big/{}_p\Psi_q\Big[_{\;\;\;\;\;(\beta_{q}, B_{q})}^{(\alpha_1^\prime, A_1),(\alpha_{p-1}, A_{p-1})}\Big|z \Big],$$
is decreasing (increasing) on $(0,\infty).$ Furthermore, the following inequality 
\begin{equation}
{}_p\Psi_q\Big[^{(\alpha_1+A_1, A_1),(\alpha_{p-1}+A_{p-1}, A_{q-1})}_{\;\;\;\;\;\;\;(\beta_q+B_q, B_q)}\Big|z \Big]{}_p\Psi_q\Big[^{(\alpha_1^\prime, A_1),(\alpha_{p-1}, A_{p-1})}_{\;\;\;\;\;(\beta_q, B_q)}\Big|z \Big]\leq(\geq){}_p\Psi_q\Big[^{(\alpha_1^\prime+A_1, A_1),(\alpha_{p-1}+A_{p-1}, A_{q-1})}_{\;\;\;\;\;(\beta_q+B_p, B_p)}\Big|z \Big]{}_p\Psi_q\Big[^{(\alpha_p,A_p)}_{(\beta_q, B_q)}\Big|z \Big], 
\end{equation}
holds true for all $z>0.$ Letting $A=B=1$ in the above inequality, we obtain the following inequality for the hypergeometric function $_pF_q$
\begin{equation}
{}_pF_q\Big[^{\alpha_1+1,...,\alpha_p+1}_{\beta_1+1,...,\beta_q+1}\Big|z \Big]{}_pF_q\Big[^{\alpha_1^\prime,...,\alpha_p}_{\beta,...,\beta_q}\Big|z \Big]\leq(\geq){}_pF_q\Big[^{\alpha_1^\prime+1,...,\alpha_p+1}_{\beta_1+1,...,\beta_q+1}\Big|z \Big]{}_pF_q\Big[^{\alpha_1,...,\alpha_p}_{\beta,...,\beta_q}\Big|z \Big],
\end{equation}

\end{remark} 

\begin{theorem} Let $\alpha,\beta>0, A, B\geq0$ and $n\in\mathbb{N},$ we define the function ${}_p\Psi_q^n$ by
\begin{equation*}
\begin{split}
{}_p\Psi_q^n\Big[_{(\beta_q,B_q)}^{(\alpha_p,A_p)}\Big|z \Big]&={}_p\Psi_q\Big[_{(\beta_q,B_q)}^{(\alpha_p,A_p)}\Big|z \Big]-\sum_{k=0}^n\frac{\prod_{j=1}^p\Gamma(\alpha_j+kA_j)z^k}{k!\prod_{j=1}^q\Gamma(\beta_j+kB_j)}\\
&=\sum_{k=n+1}^\infty\frac{\prod_{j=1}^p\Gamma(\alpha_j+kA_j)z^k}{k!\prod_{j=1}^q\Gamma(\beta_j+kB_j)}
\end{split}
\end{equation*}
Then, the following Tur\'an type inequality
\begin{equation}\label{MM}
\left({}_p\Psi_q^{n+1}\Big[_{(\beta_p,B_q)}^{(\alpha_1,0)}\Big|z \Big]\right)^2-{}_p\Psi_q^n\Big[_{(\beta_q,B_q)}^{(\alpha_p,0)}\Big|z \Big]{}_p\Psi_q^{n+2}\Big[_{(\beta_q,B_q)}^{(\alpha_p, 0)}\Big|z \Big]\geq0,
\end{equation}
is valid for all $z\in(0,\infty).$
\end{theorem}
\begin{proof} By taking into account the obvious equations :
$${}_p\Psi_q^{n}\Big[_{(\beta_q,B_q)}^{(\alpha_p,0)}\Big|z \Big]={}_p\Psi_q^{n+1}\Big[_{(\beta_q,B_q)}^{(\alpha_p,0)}\Big|z \Big]+\frac{\prod_{j=1}^p\Gamma(\alpha_j)z^{n+1}}{(n+1)!\prod_{j=1}^q\Gamma(\beta_j+(n+1)B_j)}$$
and
$${}_p\Psi_q^{n+2}\Big[_{(\beta_q,B_q)}^{(\alpha_p,0)}\Big|z \Big]={}_p\Psi_q^{n+1}\Big[_{(\beta_q,B_q)}^{(\alpha_p,0)}\Big|z \Big]-\frac{\prod_{j=1}^p\Gamma(\alpha_j)z^{n+2}}{(n+2)!\prod_{j=1}^q\Gamma(\beta_j+(n+2)B_j)}.$$
We get
$$\left({}_p\Psi_q^{n+1}\Big[_{(\beta_q,B_q)}^{(\alpha_p,0)}\Big|z \Big]\right)^2-{}_p\Psi_q^n\Big[_{(\beta_q,B_q)}^{(\alpha_p,0)}\Big|z \Big]{}_p\Psi_q^{n+2}\Big[_{(\beta_q,B_q)}^{(\alpha_p, 0)}\Big|z \Big]=$$
\begin{equation*}
\begin{split}
\:\:\:\:\:\:\:\:\:\:\:\:&={}_p\Psi_q^{n+1}\Big[_{(\beta_q,B_q)}^{(\alpha_p,0)}\Big|z \Big]\left[\frac{\prod_{j=1}^p\Gamma(\alpha_j)z^{n+2}}{(n+2)!\prod_{j=1}^q\Gamma(\beta_j+(n+2)B_j)}-\frac{\prod_{j=1}^p\Gamma(\alpha_j)z^{n+1}}{(n+1)!\prod_{j=1}^q\Gamma(\beta_j+(n+1)B_j)}\right]\\
&+\frac{\prod_{j=1}^p\Gamma^2(\alpha_j)z^{2n+3}}{(n+2)!(n+1)!\prod_{j=1}^q\Gamma(\beta_j+(n+1)B_j)\Gamma(\beta_j+(n+2)B_j)}\\
&=\frac{\prod_{j=1}^p\Gamma^2(\alpha_j)z^{n+2}}{(n+2)!\prod_{j=1}^q\Gamma(\beta_j+(n+2)B_j)}\sum_{k=n+2}^\infty\frac{z^{k}}{k!\prod_{j=1}^q\Gamma(\beta_j+kB_j)}-\frac{\prod_{j=1}^p\Gamma^2(\alpha_j)z^{n+1}}{(n+1)!\prod_{j=1}^q\Gamma(\beta_j+(n+1)B_j)}\\
&\times \sum_{k=n+3}^\infty\frac{z^{k}}{k!\prod_{j=1}^q\Gamma(\beta_j+kB_j)}\\
&=\frac{\prod_{j=1}^p\Gamma^2(\alpha_j)z^{n+2}}{(n+2)!\prod_{j=1}^q\Gamma(\beta_j+(n+2)B_j)}\sum_{k=n+3}^\infty\frac{z^{k-1}}{(k-1)!\prod_{j=1}^q\Gamma(\beta_j+(k-1)B_j)}-\frac{ \prod_{j=1}^p\Gamma^2(\alpha_j)z^{n+1}}{(n+1)!\prod_{j=1}^q\Gamma(\beta_j+(n+1)B_j)}\\
&\times \sum_{k=n+3}^\infty\frac{z^{k}}{k!\prod_{j=1}^q\Gamma(\beta_j+kB_j)}\\
&=\sum_{k=n+3}^\infty\frac{\prod_{j=1}^p\Gamma^2(\alpha_j)\Delta_{n,k}(\beta, B) z^{k+n+1}}{k!(k-1)!(n+1)!(n+2)!\prod_{j=1}^q\Gamma(\beta_j+kB_j)\Gamma(\beta_j+(k-1)B_j)\Gamma(\beta_j+(n+1)B_j)\Gamma(\beta_j+(n+2)B_j)}
\end{split}
\end{equation*}
where $\Delta_{n,k}(\beta, B)$ is defined for all $k\geq n+3$ by
\begin{equation*}
\begin{split}
\Delta_{n,k}(\beta, B)&=(n+1)!k!\prod_{j=1}^q\Gamma(\beta_j+kB_j)\Gamma(\beta_j+(n+1)B_j)-(n+2)!(k-1)!\prod_{j=1}^q\Gamma(\beta_j+(k-1)B_j)\Gamma(\beta_j+(n+2)B_j)\\
&\geq(n+2)!(k-1)!\Big(\prod_{j=1}^q\Gamma(\beta_j+kB_j)\Gamma(\beta_j+(n+1)B_j)-\prod_{j=1}^q\Gamma(\beta_j+(k-1)B_j)\Gamma(\beta_j+(n+2)B_j)\Big).
\end{split}
\end{equation*}
Now, let $z=\beta_j+(n+1)B_j,\;a=B_j$ and $b=B_j(k-(n+2))$ in (\ref{32}) we deduce that 
\begin{equation*}
\Gamma(\beta_j+kB_j)\Gamma(\beta_j+(n+1)B_j)\geq\Gamma(\beta_j+(k-1)B_j)\Gamma(\beta_j+(n+2)B_j).
\end{equation*}
The desired inequality (\ref{MM}) is thus established.
\end{proof}

\begin{theorem}\label{T6}Let $\alpha,\beta>0, A, B\geq0$ and $n\in\mathbb{N}.$ We define the function $\mathcal{K}_n^{(\alpha,\beta)}(A, B, z)$ by
\begin{equation}\label{RTR}
\mathcal{K}_n^{(\alpha,\beta)}(A, B, z)=\frac{{}_p\Psi_q^n\Big[_{(\beta_q,B_q)}^{(\alpha_p, A_p)}\Big|z \Big]{}_p\Psi_q^{n+2}\Big[_{(\beta_q,B_q)}^{(\alpha_p, A_p)}\Big|z \Big]}{\left({}_p\Psi_q^{n+1}\Big[_{(\beta_q,B_q)}^{(\alpha_p,A_p)}\Big|z \Big]\right)^2}.
\end{equation}
Then the function $z\mapsto\mathcal{K}_n^{(\alpha,\beta)}(0, B, z)$ is increasing on $(0,\infty).$ Moreover, the following Tur\'an type inequality 
$$\Bigg(\frac{n+2}{n+3}\Bigg).\left(\frac{\prod_{j=1}^q\Gamma^2(\beta_j+(n+2)B_j)}{\prod_{j=1}^q\Gamma(\beta_j+(n+1)B_j)\Gamma(\beta_j+(n+3)B_j)}\right)\left({}_p\Psi_q^{n+1}\Big[_{(\beta_q,B_q)}^{(\alpha_p,0)}\Big|z \Big]\right)^2$$
\begin{equation}\label{MMM}
\begin{split}
&\leq{}_p\Psi_q^n\Big[_{(\beta_q,B_q)}^{(\alpha_p,0)}\Big|z \Big]{}_p\Psi_q^{n+2}\Big[_{(\beta_q,B_q)}^{(\alpha_p, 0)}\Big|z \Big],
\end{split}
\end{equation}
holds for all $\alpha,\beta>0,\;n\in\mathbb{N}$ and $z\in(0,\infty).$ The constant in LHS of inequality (\ref{MMM}) is sharp.
\end{theorem}
\begin{proof}By applying the Cauchy product, we find that
\begin{equation*}
\mathcal{K}_n^{(\alpha,\beta)}(0, B, z)=\sum_{k=0}^\infty\sum_{i=0}^k U_i^{1}(\alpha,\beta, B) z^{k}\Big/\sum_{k=0}^\infty\sum_{i=0}^k V_i^{1}(\alpha,\beta, B) z^{k},
\end{equation*}
where 
$$U_i^{1}(\alpha,\beta, B)=\frac{\prod_{j=1}^p\Gamma^2(\alpha_j)}{(i+n+1)!(k-i+n+3)!\prod_{j=1}^q\Big(\Gamma(\beta_j+(i+n+1)B_j)\Gamma(\beta_j+(k-i+n+3)B_j)\Big)}$$
and 
$$V_i^{1}(\alpha,\beta, B)=\frac{\prod_{j=1}^p\Gamma^2(\alpha_j)}{(i+n+2)!(k-i+n+2)!\prod_{j=1}^q\Big(\Gamma(\beta_j+(i+n+2)B_j)\Gamma(\beta_j+(k-i+n+2)B_j)\Big)}.$$
Next, we define the sequence $(W_i^{1}(\alpha,\beta, B)=U_i^{1}(\alpha,\beta, B)/V_i^{1}(\alpha,\beta, B))_{i\geq0}.$ Thus
\begin{equation}\label{R}
\begin{split}
\frac{W_{i+1}^{1}(\alpha,\beta, B)}{W_i^{1}(\alpha,\beta, B)}&=\frac{(i+n+2)(k-i+n+2)}{(i+n+1)(k-i+n+1)}\\
&\times\frac{\prod_{j=1}^q\Gamma(\beta_j+(i+n+1)B_j)\Gamma(\beta_j+(k-i+n+1)B_j)\Gamma(\beta_j+(i+n+3)B_j)\Gamma(\beta_j+(k-i+n+3)B_j)}{\prod_{j=1}^q\Big(\Gamma^2(\beta_j+(i+n+2)B_j)\Gamma^2(\beta_j+(k-i+n+2)B_j)\Big)}\\
&\geq
\frac{\prod_{j=1}^q\Gamma(\beta_j+(i+n+1)B_j)\Gamma(\beta_j+(k-i+n+1)B_j)\Gamma(\beta_j+(i+n+3)B_j)\Gamma(\beta_j+(k-i+n+3)B_j)}{\prod_{j=1}^q\Big(\Gamma^2(\beta_j+(i+n+2)B_j)\Gamma^2(\beta_j+(k-i+n+2)B_j)\Big)}\\
&=\Bigg(\frac{\prod_{j=1}^q\Gamma(\beta_j+(i+n+1)B_j)\Gamma(\beta_j+(i+n+3)B_j)}{\prod_{j=1}^q\Big(\Gamma^2(\beta_j+(i+n+2)B_j)}\Bigg)\\&\times\Bigg(\frac{\prod_{j=1}^q \Gamma(\beta_j+(k-i+n+1)B_j)\Gamma(\beta_j+(k-i+n+3)B_j)}{\prod_{j=1}^q\Gamma^2(\beta_j+(k-i+n+2)B_j)}\Bigg).
\end{split}
\end{equation} 
Let $z=\beta_j+(i+n+1)B_j$ and $a=b=B_j$ in (\ref{32}) we deduce that 
\begin{equation}\label{RR}
\Gamma(\beta_j+(i+n+1)B_j)\Gamma(\beta_j+(i+n+3)B_j)\geq \Gamma^2(\beta_j+(i+n+2)B_j).
\end{equation}
Upon replacing $i$ by $k-i$ in (\ref{RR}), we obtain 
\begin{equation}\label{RRr}
\Gamma(\beta_j+(k-i+n+1)B_j)\Gamma(\beta_j+(k-i+n+3)B_j)\geq\Gamma^2(\beta_j+(k-i+n+2)B_j).
\end{equation}
In view of (\ref{R}), (\ref{RR}) and (\ref{RRr}) we deduce that the sequence $(W_i^{1}(\alpha,\beta, B))_{i\geq0})$ is increasing, and  consequently $\sum_{i=0}^k U_i^{1}(\alpha,\beta, B)/\sum_{i=0}^k V_i^{1}(\alpha,\beta, B)$ is increasing by means of Lemma \ref{l1}. Hence, the function $z\mapsto\mathcal{K}_n^{(\alpha,\beta)}(0, B, z)$ is increasing on $(0,\infty)$, by Lemma \ref{l2}. Finally, since
$$\lim_{x\rightarrow 0}\mathcal{K}_n^{(\alpha,\beta)}(0, B, z)=\left(\frac{n+2}{n+3}\right).\left(\frac{\prod_{j=1}^q\Gamma^2(\beta_j+(n+2)B_j)}{\prod_{j=1}^q\Gamma(\beta_j+(n+1)B_j)\Gamma(\beta_j+(n+3)B_j)}\right),$$
and it follows that the constant 
$$\left(\frac{n+2}{n+3}\right).\left(\frac{\prod_{j=1}^q\Gamma^2(\beta_j+(n+2)B_j)}{\prod_{j=1}^q\Gamma(\beta_j+(n+1)B_j)\Gamma(\beta_j+(n+3)B_j)}\right),$$
is the best possible for which the inequality (\ref{MMM}) holds for all $\alpha,\beta>0,\;B\geq0$ and $z>0.$  With this the proof of Theorem \ref{T6}  is complete.
\end{proof}


\section{Lazarevi\'c and Wilker type inequalities for the Fox-Wright function}
\begin{theorem}\label{TT7} Let $\alpha_1,\beta>0$ and $B_1\geq0.$ If $\alpha_1\geq\beta_2$, then the function 
\begin{equation}
\beta_1\mapsto \chi(\beta_1)=\frac{{}_1\tilde{\Psi}_2\Big[_{(\beta_1+B_1, B_1), (\beta_2+1, 1)}^{\;\;\;\;\;(\alpha_1+1, 1)\;\;\;\;}\Big|z \Big]}{{}_1\tilde{\Psi}_2\Big[_{(\beta_1, B_1), (\beta_2, 1)}^{\;\;\;\;\;(\alpha_1, 1)\;\;\;\;\;}\Big|z \Big]},
\end{equation}
is increasing on $(0,\infty).$
\end{theorem}
\begin{proof} By using the fact we that the function $\beta_1\mapsto{}_2\tilde{\Psi}_2\Big[_{(\beta_1,B_1), (\beta_2,B_2)}^{(\alpha_1, A_1), (\alpha_2,A_2)}\Big|z \Big]$ is log--convex on $(0,\infty)$ ( see the proof of Theorem  \ref{T1}), and hence the function 
$$\beta_1\mapsto \log {}_2\tilde{\Psi}_2\Big[_{(\beta_1+B_1,B_1), (\beta_2,B_2)}^{(\alpha_1, A_1), (\alpha_2,A_2)}\Big|z \Big]-\log {}_2\tilde{\Psi}_2\Big[_{(\beta_1,B_1), (\beta_2,B_2)}^{(\alpha_1, A_1), (\alpha_2,A_2)}\Big|z \Big]$$ 
is increasing on $(0,\infty)$. Consequently the function 
$$
\beta_1\mapsto\phi(\beta_1)=\frac{{}_2\tilde{\Psi}_2\Big[_{(\beta_1+B_1, B_1), (\beta_2, B_2)}^{(\alpha_1, A_1),\;\;(\alpha_2, A_2)}\Big|z \Big]}{{}_2\tilde{\Psi}_2\Big[_{(\beta_1, B_1), (\beta_2, B_2)}^{(\alpha_1, A_1),\;\; (\alpha_2, A_2)}\Big|z \Big]},$$
is increasing on $(0,\infty)$ for all $z>0.$ In particular, the function 
$$\beta_1\mapsto \chi_1(\beta_1)=\frac{{}_1\tilde{\Psi}_2\Big[_{(\beta_1+B_1,B_1), (\beta_2,1)}^{\;\;\;\;\;\;(\alpha_1, 1)\;\;\;}\Big|z \Big]}{{}_1\tilde{\Psi}_2\Big[_{(\beta_1,B_1), (\beta_2,1)}^{\;\;\;(\alpha_1, 1)\;\;\;}\Big|z \Big]}$$
is increasing on $(0,\infty)$ for all $z>0.$
On the other hand, we set 
$$\chi_2(\beta_1)=\frac{{}_1\tilde{\Psi}_2\Big[_{(\beta_1+B_1, B_1), (\beta_2+1, 1)}^{\;\;\;\;\;(\alpha_1+1, 1)\;\;\;\;\;}\Big|z \Big]}{{}_1\tilde{\Psi}_2\Big[_{(\beta_1+B_1, B_1), (\beta_2, 1)}^{\;\;\;\;\;\;(\alpha_1, 1)\;\;\;\;}\Big|z \Big]}=\frac{{}_1\Psi_2\Big[_{(\beta_1+B_1, B_1), (\beta_2+1, 1)}^{\;\;\;\;\;\;(\alpha_1+1, 1)\;\;\;\;\;}\Big|z \Big]}{{}_1\Psi_2\Big[_{(\beta_1+B_1, B_1), (\beta_2, 1)}^{\;\;\;\;\;\;\;(\alpha_1, 1)\;\;}\Big|z \Big]}.$$
Then,
\begin{equation}\label{B}
\begin{split}\Omega(\beta_1)&=\left({}_1\Psi_2\Big[_{(\beta_1+B_1, B_1), (\beta_2, 1)}^{\;\;\;\;(\alpha_1, 1)\;\;\;\;\;}\Big|z \Big]\right)^2.\frac{\partial \chi_2(\beta_1)}{\partial\beta_1}\\
&=\frac{\partial}{\partial\beta_1}\left({}_1\Psi_2\Big[_{(\beta_1+B_1, B_1), (\beta_2+1, 1)}^{\;\;\;\;(\alpha_1+1, 1)\;\;\;\;\;\;}\Big|z \Big]\right).{}_1\Psi_2\Big[_{(\beta_1+B_1, B_1), (\beta_2, 1)}^{\;\;\;\;(\alpha_1,\;1)\;\;\;\;\;}\Big|z \Big]
-\frac{\partial}{\partial\beta_1}\left({}_1\Psi_2\Big[_{(\beta_1+B_1, B_1), (\beta_2, 1)}^{\;\;\;\;(\alpha_1, 1)\;\;\;\;\;}\Big|z \right)\Big]\\&\times{}_1\Psi_2\Big[_{(\beta_1+B_1, B_1), (\beta_2+1,\;1)}^{\;\;\;\;(\alpha_1+1,\;1)\;\;\;\;\;}\Big|z \Big].
\end{split}
\end{equation}
Moreover, we have
\begin{equation}
\frac{\partial}{\partial\beta_1}{}_1\Psi_2\Big[_{(\beta_1+B_1,\;B_1), (\beta_2+1,\;1)}^{\;\;\;\;(\alpha_1+1,\;1)\;\;\;\;\;}\Big|z \Big]=-\sum_{k=0}^\infty\frac{\psi(\beta_1+B_1+kB_1)\Gamma(\alpha_1+k+1)}{k!\Gamma(\beta_1+B_1+kB_1)\Gamma(\beta_2+k+1)}z^k,
\end{equation}
and 
\begin{equation}
\frac{\partial}{\partial\beta_1}\left({}_1\Psi_2\Big[_{(\beta_1+B_1,\;B_1), (\beta_2,\;1)}^{\;\;\;\;\;\;\;\;(\alpha_1,\;1)\;\;\;\;\;}\Big|z \right)\Big]=-\sum_{k=0}^\infty\frac{\psi(\beta_1+B_1+kB_1)\Gamma(\alpha_1+k)}{k!\Gamma(\beta_1+B_1+kB_1)\Gamma(\beta_2+k)}z^k.
\end{equation}
By applying the Cauchy product, we find that
\begin{equation}\label{BB}
\left(\frac{\partial}{\partial\beta_1}{}_1\Psi_2\Big[_{(\beta_1+B_1,\;B_1), (\beta_2+1,\;1)}^{\;\;\;\;(\alpha_1+1, \;1)\;\;\;\;\;}\Big|z \Big]\right).{}_1\Psi_2\Big[_{(\beta_1+B_1, B_1), (\beta_2, 1)}^{\;\;\;\;(\alpha_1, 1)\;\;\;\;\;}\Big|z \Big]=-\sum_{k=0}^\infty\sum_{j=0}^k\frac{\Omega_{j,k}\psi(\beta_1+B_1+jB_1)(\alpha_1+j)z^k}{(\beta_2+j)}
\end{equation}
and
\begin{equation}\label{BBB}
\frac{\partial}{\partial\beta_1}\left({}_1\Psi_2\Big[_{(\beta_1+B_1,\;B_1), (\beta_2, 1)}^{\;\;\;\;\;\;(\alpha_1, 1),\;\;\;\;\;}\Big|z \right)\Big].{}_1\Psi_2\Big[_{(\beta_1+B_1, B_1), (\beta_2+1, 1)}^{\;\;\;\;\;(\alpha_1+1,\;1)\;\;\;\;\;}\Big|z \Big]=-\sum_{k=0}^\infty\sum_{j=0}^k\frac{\Omega_{j,k}\psi(\beta_1+B_1+jB_1)(\alpha_1+(k-j))z^k}{(\beta_2+(k-j))}.
\end{equation}
where 
$$\Omega_{j,k}=\frac{\Gamma(\alpha_1+j)\Gamma(\alpha_1+(k-j))}{j!(k-j)\Gamma(\beta_1+B_1+jB_1)\Gamma(\beta_1+B_1+(k-j)B_1)\Gamma(\beta_2+j)\Gamma(\beta_2+(k-j))}$$
In view of (\ref{B}), (\ref{BB}) and (\ref{BBB}), we obtain 
\begin{equation}
\begin{split}
\Omega(\beta_1)&=\sum_{k=0}^\infty\sum_{j=0}^k\Omega_{j,k}\psi(\beta_1+B_1+jB_1)\left[\frac{\alpha_1+(k-j)}{\beta_2+k-j}-\frac{\alpha_1+j}{\beta_2+j}\right]z^k\\
&=\sum_{k=0}^\infty\sum_{j=0}^{[(k-1)/2]}\Omega_{j,k}\frac{(k-2j)(\alpha_1-\beta_2)\Big(\psi(\beta_1+B_1+(k-j)B_1)-\psi(\beta_1+B_1+jB_1)\Big)}{(\beta_2+k-j)(\beta_2+j)}.
\end{split}
\end{equation}
From the fact that the digamma function  $\psi$ is increasing on $(0,\infty)$ we deduce for $k-j>j$ (i.e $[(k-1)/2]\geq j$), 
$$\psi(\beta_1+B_1+(k-j)B_1)-\psi(\beta_1+B_1+jB_1)>0,$$
and  $k-2j>0.$ Hence the function $\Omega(\beta_1)$ is positive under the conditions stated. Furthermore, the function $\beta_1\mapsto\chi_2(\beta_1)$ is increasing on $(0,\infty).$ So the function $\chi(\beta_1)=\chi_1(\beta_1)\chi_2(\beta_1)$ is increasing on $(0,\infty),$ as a product of two positive and increasing functions.
\end{proof}
\begin{theorem}\label{T7} Let $\alpha_1, \beta>0,$ such that $B_1\geq0.$ If $\alpha_1\geq\beta_2.$ Then the following inequality
\begin{equation}\label{Y}
\Bigg[\Bigg(\frac{\Gamma(\alpha_1)}{\Gamma(\beta_2)}\Bigg)^{\frac{B_1}{\beta_1}}\times{}_1\tilde{\Psi}_2\Big[_{(\beta_1,\;B_1), (\beta_2,\;1)}^{\;\;\;\;\;(\alpha_1, 1)\;\;\;\;\;}\Big|z \Big]\Bigg]^{\frac{\Gamma(\beta_1+B_1)}{\Gamma(\beta_1)}}\leq \Bigg[{}_1\tilde{\Psi}_2\Big[_{(\beta_1+1, B_1), (\beta_2, 1)}^{\;\;\;\;\;(\alpha_1,\;1)\;\;\;\;}\Big|z \Big]\Bigg]^{\frac{\Gamma(\beta_1+B_1+1)}{\Gamma(\beta_1+1)}}
\end{equation}
holds true for all $z\in(0,\infty).$
\end{theorem}
\begin{proof}Suppose that $\alpha_1\geq \beta_2$ and we define the function $\Xi:(0,\infty)\longrightarrow\mathbb{R}$ with the following relation:
\begin{equation*}
\Xi(z)=\frac{\beta_1+B_1}{\beta_1}\log \Bigg[{}_1\tilde{\Psi}_2\Big[_{(\beta_1+1, B_1), (\beta_2, 1)}^{\;\;\;\;\;\;(\alpha_1, 1)\;\;\;}\Big|z \Big]\Bigg]-\log \Bigg[{}_1\tilde{\Psi}_2\Big[_{(\beta_1, B_1), (\beta_2, 1)}^{\;\;\;\;\;\;(\alpha_1, 1)\;\;\;}\Big|z \Big]\Bigg].
\end{equation*}
Make use of the following formula 
 $$\Bigg[{}_1\tilde{\Psi}_2\Big[_{(\beta_1,B_1), (\beta_2, 1)}^{\;\;\;\;\;\;(\alpha_1, 1)\;\;\;}\Big|z \Big]\Bigg]^\prime=\frac{\Gamma(\beta_1)}{\Gamma(\beta_1+B_1)}{}_1\tilde{\Psi}_2\Big[_{(\beta_1+B_1,B_1),(\beta_2+1, 1)}^{\;\;\;\;\;\;(\alpha_1+1, 1)\;\;\;}\Big|z \Big],$$
we thus get
\begin{equation}\label{N0}
\begin{split}
\Xi^\prime(z)&=\frac{\Gamma(\beta_1)}{\Gamma(\beta_1+B_1)}\left(\frac{{}_1\tilde{\Psi}_2\Big[_{(\beta_1+B_1+1,B_1), (\beta_2+1,1)}^{\;\;\;\;\;\;(\alpha_1+1,\;1)\;\;\;}\Big|z \Big]}{{}_1\tilde{\Psi}_2\Big[_{(\beta_1+1,B_1), (\beta_2, 1)}^{\;\;\;\;\;\;(\alpha_1, 1)\;\;\;}\Big|z \Big]}-\frac{{}_1\tilde{\Psi}_2\Big[_{(\beta_1+B_1,B_1), (\beta_2+1, 1)}^{\;\;\;\;\;\;(\alpha_1+1, 1)\;\;\;}\Big|z \Big]}{{}_1\tilde{\Psi}_2\Big[_{(\beta_1, B_1), (\beta_2, 1)}^{\;\;\;\;\;\;(\alpha_1, 1)\;\;\;}\Big|z \Big]}\right)\\
&=\frac{\Gamma(\beta_1)}{\Gamma(\beta_1+B_1)}\Big(\chi(\beta_1+1)-\chi(\beta_1)\Big).
\end{split}
\end{equation}
 By  by taking into account the Theorem \ref{TT7} we deduce that $\Xi^\prime(z)\geq0$, and consequently the function $\Xi(z)$ is increasing on $(0,\infty).$ Hence 
\begin{equation}\label{tyur}
\Xi(z)\geq\Xi(0)=\frac{B_1}{\beta_2}\log\Big(\frac{\Gamma(\alpha_1)}{\Gamma(\beta_2)}\Big).
\end{equation}
 By these observation and using the relationship :
$$\frac{\beta_1+B_1}{\beta_1}=\left(\frac{\Gamma(\beta_1+B_1+1)}{\Gamma(\beta_1+1)}\right).\left(\frac{\Gamma(\beta_1)}{\Gamma(\beta_1+B_1)}\right),$$
we can complete the proof of the above-asserted results immediately.
\end{proof}

\begin{coro}Let $\alpha, \beta>0,$ such that $\alpha_1\geq\beta_2.$ Then the following inequality 
\begin{equation}\label{ààà}
\frac{{}_1\tilde{\Psi}_2\Big[_{(\beta_1+1, B_1), (\beta_2, 1)}^{\;\;\;\;(\alpha_1, 1)\;\;\;}\Big|z \Big]}{{}_1\tilde{\Psi}_2\Big[_{(\beta_1, B_1), (\beta_2, 1)}^{\;\;\;\;(\alpha_1, 1)\;\;\;}\Big|z \Big]}+\Bigg[\frac{\Gamma(\beta_2)}{\Gamma(\alpha_1)}\times{}_1\tilde{\Psi}_2\Big[_{(\beta_1+1, B_1),\;\;\;\; (\beta_2, 1)}^{\;\;\;\;(\alpha_1, 1)\;\;\;}\Big|z \Big]\Bigg]^{\frac{B_1}{\beta_1}}\geq 2,
\end{equation}
is valid for all $z\in(0,\infty).$
\end{coro}
\begin{proof}From the inequality (\ref{tyur}), we have
\begin{equation*}\label{KKK}\frac{\Bigg[{}_1\tilde{\Psi}_2\Big[_{(\beta_1+1,B_1), (\beta_2, 1)}^{\;\;\;\;(\alpha_1, 1)\;\;\;\;}\Big|z \Big]\Bigg]^{\frac{\beta_1+B_1}{\beta_1}}}{\left[\frac{\Gamma(\alpha_1)}{\Gamma(\beta_2)}\right]^{\frac{B_1}{\beta_1}} {}_1\tilde{\Psi}_2\Big[_{(\beta_1, B_1), (\beta_2, 1 )}^{\;\;\;\;(\alpha_1, 1)\;\;\;\;}\Big|z \Big]}=\frac{{}_1\tilde{\Psi}_2\Big[_{(\beta_1+1, B_1), (\beta_2, 1)}^{\;\;\;\;(\alpha_1, 1)\;\;\;\;}\Big|z \Big]}{{}_1\tilde{\Psi}_2\Big[_{(\beta_1, B_1), (\beta_2, 1)}^{\;\;\;\;(\alpha_1, 1)\;\;\;\;}\Big|z \Big]}.\Bigg[\frac{\Gamma(\beta_2)}{\Gamma(\alpha_1)}\times{}_1\tilde{\Psi}_2\Big[_{(\beta_1+1, B_1),(\beta_2, 1)}^{\;\;\;\;(\alpha_1, 1)\;\;\;\;}\Big|z \Big]\Bigg]^{\frac{B_1}{\beta_1}}\geq1.
\end{equation*}
 If we use the above inequality  and the Arithmetic-Geometric Mean Inequality, we find that
$$\frac{1}{2}\left[\frac{{}_1\tilde{\Psi}_2\Big[_{(\beta_1+1, B_1), (\beta_2, 1)}^{\;\;\;\;(\alpha_1, 1)\;\;\;\;}\Big|z \Big]}{{}_1\tilde{\Psi}_2\Big[_{(\beta_1, B_1), (\beta_2, 1)}^{\;\;\;\;(\alpha_1, 1)\;\;\;\;}\Big|z \Big]}+\Bigg[\frac{\Gamma(\beta_2)}{\Gamma(\alpha_1)}{}_1\tilde{\Psi}_2\Big[_{(\beta_1+1, B_1),(\beta_2, 1)}^{\;\;\;\;(\alpha_1, 1)\;\;\;\;}\Big|z \Big]\Bigg]^{\frac{B_1}{\beta_1}}\right]$$
\begin{equation}
\begin{split}
\;\;\;\;\;\;\;\;\;\;\;\;\;\;\;\;\;\;\;\;\;\;\;\;\;\;\;\;\;\;\;\;\;\;\;\;\;\;\;\;\;\;\;\;\;\;\;\;\;\;\;\;&\geq \sqrt{\frac{\Bigg[{}_1\tilde{\Psi}_2\Big[_{(\beta_1+1,B_1), (\beta_2, 1)}^{\;\;\;\;(\alpha_1, 1)\;\;\;\;}\Big|z \Big]\Bigg]^{\frac{\beta_1+B_1}{\beta_1}}}{\left[\frac{\Gamma(\alpha_1)}{\Gamma(\beta_2)}\right]^{\frac{B_1}{\beta_1}}\times{}_1\tilde{\Psi}_2\Big[_{(\beta_1, B_1), (\beta_2, 1 )}^{\;\;\;\;(\alpha_1, 1)\;\;\;\;}\Big|z \Big]}}\\
&\geq 1.
\end{split}
\end{equation}
This completes the proof.
\end{proof}

Letting in the inequalities (\ref{Y}) and (\ref{ààà}) the value $B_1=1,$ we obtain the Lazarevi\'c and Wilker type inequalities for the hypergeometric function ${}_1F_2.$

\begin{coro} Let $\alpha_1,\beta>0.$ Then the following inequalities 
\begin{equation}
\Big[{}_1F_2\Big(^{\;\;\;\alpha_1}_{\beta_1,\;\beta_2}\Big| z\Big)\Big]^{\beta_1}\leq \Big[{}_1F_2\Big(^{\;\;\;\;\alpha_1}_{\beta_1+1,\;\beta_2}\Big| z\Big)\Big]^{\beta_1+1},
\end{equation}
and
\begin{equation}
\frac{{}_1F_2\Big(^{\;\;\;\;\alpha_1}_{\beta_1+1,\;\beta_2}\Big| z\Big)}{{}_1F_2\Big(^{\;\;\;\alpha_1}_{\beta_1,\;\beta_2}\Big| z\Big)}+\Big[{}_1F_2\Big(^{\;\;\;\;\alpha_1}_{\beta_1+1,\;\beta_2}\Big| z\Big)\Big]^{\frac{1}{\beta_1}}\geq2,
\end{equation}
holds true for all $z\in(0,\infty).$
\end{coro}

Letting $\alpha_1=1$ in the inequalities (\ref{Y}) and (\ref{ààà}), we get the  Lazarevi\'c and Wilker type inequalities  for the four-parametric Mittag--Leffler function $E_{B_1,\beta_1;1,\beta_2}(z).$

\begin{coro} Let $\beta_1>0$ and $B_1\geq0.$ If $0<\beta_2\geq1,$ then the following inequalities
\begin{equation}
\left[\left(\frac{1}{\Gamma(\beta_2)}\right)^{\frac{B_1}{\beta_1}}E_{B_1,\beta_1;1,\beta_2}(z)\right]^{\frac{\Gamma(\beta_1+B_1)}{\Gamma(\beta_1)}}\leq \Big[E_{B_1,\beta_1+1;1,\beta_2}(z)\Big]^{\frac{\Gamma(\beta_1+B_1+1)}{\Gamma(\beta_1+1)}},
\end{equation}
and
\begin{equation}
\frac{E_{B_1,\beta_1+1;1,\beta_2}(z)}{E_{B_1, \beta_1;1,\beta_2}(z)}+\Big[\Gamma(\beta_2)E_{B_1, \beta_1+1;1, \beta_2}(z)\Big]^{\frac{B_1}{\beta_1}}\geq2,
\end{equation}
holds true for all $z>0.$
\end{coro} 
\vspace{0,3cm}
\begin{remark}
\noindent 1. Letting $B_1=1$ in Theorem \ref{TT7}, we conclude that  the function 
$$\beta_1\mapsto {}_1F_2\left(^{\;\;\;\alpha_1+1}_{\beta_1+1,\;\beta_2+1}\Big|z \right)\Big/{}_1F_2\left(^{\;\;\;\alpha_1}_{\beta_1,\;\beta_2}\Big|z \right),$$
is increasing on $(0,\infty).$\\
\vspace{0,3cm}
\noindent 2. If we choose $\alpha_1=\beta_2,$  in (\ref{Y}) (resp. in (\ref{ààà})), we conclude that the inequality (\ref{Y}) (resp.  (\ref{ààà}) is a natural  generalization of the Lazarevi\'c type inequality for the Wright function \cite[Theorem 4.1, p. 138]{Khaled1}
\begin{equation}\label{211}
\left[\mathcal{W}_{B_1,\beta_1}(z)\right]^{\frac{\Gamma(\beta_1+B_1)}{\Gamma(\beta_1)}}\leq\left[\mathcal{W}_{B_1,\beta_1+1}(z)\right]^{\frac{\Gamma(\beta_1+B_1+1)}{\Gamma(\beta_1+1)}}.
\end{equation} 
\noindent 3. Choosing $B_1=1$ and $\beta_1=\nu+1$ where $\nu>-1$ in (\ref{211}), we obtain  \cite[Theorem 1]{A2}:
\begin{equation}
[\mathcal{I}_\nu(z)]^{(\nu+1)/(\nu+2)}\leq\mathcal{I}_{\nu+1}(z),
\end{equation}  
where $z\in\mathbb{R}.$ It is worth mentioning that in particular we have $\mathcal{I}_{-1/2}(z)=\cosh z$ and $\mathcal{I}_{1/2}(z)=\sinh z/z$, thus if $\nu=-1/2,$ we  derive the Lazarevi\'c--type inequality [\cite[p.270]{MI}: 
$$\cosh z\leq \left(\frac{\sinh z}{z}\right)^3.$$ 
\noindent 4. If we choose $\alpha_1=\beta_2,$  in (\ref{ààà}), we deduce that the inequality (\ref{ààà}) is a natural  generalization of the Wilker type inequality for the Wright function 
\begin{equation}\label{23}
\end{equation}
\noindent 5. Taking in (\ref{23}) the values $\alpha=1$ and $\beta=\nu+1$ where $\nu>-1$, we obtain the following inequality [\cite{A2}, Theorem 1]:
\begin{equation}
\frac{\mathcal{I}_{\nu+1}(z)}{\mathcal{I}_{\nu}(z)}+[\mathcal{I}_{\nu+1}(z)]^{1/(\nu+1)}\geq2,
\end{equation}
\vspace{0.3cm}
where $z\in\mathbb{R}$. If $\nu=-1/2,$ we  derive the Wilker--type inequality \cite{WI, ZH}:
\begin{equation}
\left(\frac{\sinh z}{z}\right)^2+\frac{\tanh z}{z}\geq2,
\end{equation}
where $z\in\mathbb{R}$.

\end{remark}
\section{Further results}

In this section we show other inequalities for the Fox-Wright functions.

\begin{theorem}\label{ttt34}Let $\alpha, \beta>0,$ such that $\alpha_i\geq\beta_{i+1},\;i=1,...,p.$ Then, the function $z\mapsto {}_p\Psi_{p+1}^*\Big[_{(\beta_1, B_1),(\beta_p, 1)}^{\;\;\;\;(\alpha_p, 1)\;\;}\Big|z\Big]$ is log-concave on $(0,\infty).$ Furthermore, the following inequalities
\begin{equation}\label{876}
{}_p\Psi_{p+1}^*\Big[_{(\beta_1, B_1),(\beta_p, 1)}^{\;\;\;\;(\alpha_p, 1)\;\;}\Big|z_1\Big]{}_p\Psi_{p+1}^*\Big[_{(\beta_1, B_1),(\beta_p, 1)}^{\;\;\;\;(\alpha_p, 1)\;\;}\Big|z_2\Big]\leq{}_p\Psi_{p+1}^*\Big[_{(\beta_1, B_1),(\beta_p, 1)}^{\;\;\;(\alpha_p, 1)\;\;}\Big|\frac{z_1+z_2}{2}\Big],\;z_1, z_2>0,
\end{equation} 
\begin{equation}\label{877}
{}_p\Psi_{p+1}^*\Big[_{(\beta_1, B_1),(\beta_p, 1)}^{\;\;\;\;(\alpha_p, 1)\;\;}\Big|z\Big]\leq e^{\left(\prod_{i=1}^p\frac{\alpha_i}{\beta_{i+1}}\right)\left(\frac{\Gamma(\beta_1)}{\Gamma(\beta_1+B_1)}\right)z}, z>0.
\end{equation}
\begin{equation}\label{!!!}
{}_p\Psi_{p+1}\Big[_{(\beta_1+B_1, B_1),(\beta_p+1, 1)}^{\;\;\;\;(\alpha_p+1, 1)\;\;}\Big|z\Big]\leq\left(\prod_{i=1}^p\frac{\alpha_i}{\beta_{i+1}}\right)\left(\frac{\Gamma(\beta_1)}{\Gamma(\beta_1+B_1)}\right){}_p\Psi_{p+1}\Big[_{(\beta_1, B_1),(\beta_p, 1)}^{\;\;\;\;(\alpha_p, 1)\;\;}\Big|z\Big] , z>0.
\end{equation}
holds true.  
\end{theorem}
\begin{proof} To proved that the function he function $z\mapsto {}_p\Psi_{p+1}^*\Big[_{(\beta_1, B_1),(\beta_p, 1)}^{\;\;\;\;(\alpha_p, 1)\;\;}\Big|z\Big]$ is log-concave on $(0,\infty),$ it suffices to prove that the logarithmic derivative of ${}_p\Psi_{p+1}^*\Big[_{(\beta_1, B_1),(\beta_p, 1)}^{\;\;\;\;(\alpha_p, 1)\;\;}\Big|z\Big]$ is decreasing on $(0,\infty)$. Making use the power-series of the normalized Fox-Wright function, we get 
$$\frac{\left({}_p\Psi_{p+1}^*\Big[_{(\beta_1, B_1),(\beta_p, 1)}^{\;\;\;\;(\alpha_p, 1)\;\;}\Big|z_1\Big]\right)^\prime}{{}_p\Psi_{p+1}^*\Big[_{(\beta_1, B_1),(\beta_p, 1)}^{\;\;\;\;(\alpha_p, 1)\;\;}\Big|z\Big]}=\sum_{k=0}^\infty\frac{\prod_{i=1}^p\Gamma(\alpha_i+k+1)z^k}{k!\Gamma(\beta_1+(k+1)B_1)\prod_{i=2}^{p+1}\Gamma(\beta_i+k+1)}\Big/\sum_{k=0}^\infty\frac{\prod_{i=1}^p\Gamma(\alpha_i+k)z^k}{k!\Gamma(\beta_1+kB_1)\prod_{i=2}^{p+1}\Gamma(\beta_i+k)}.$$
Now, we define the sequence $(u_k)_{k\geq0}$ by $u_k=\left(\prod_{i=1}^p\frac{(\alpha_i+k)}{(\beta_{i+1}+k)}\right).\left(\frac{\Gamma(\beta_1+kB_1)}{\Gamma(\beta_1+(k+1)B_1)}\right).$ Thus,
\begin{equation}\label{444}
\begin{split}
\frac{u_{k+1}}{u_k}&=\left(\prod_{i=1}^p\frac{(\alpha_i+k+1)(\beta_{i+1}+k)}{(\alpha_i+k)(\beta_i+k+1)}\right).\left(\frac{\Gamma^2(\beta_1+(k+1)B_1)}{\Gamma(\beta_1+kB_1)\Gamma(\beta_1+(k+2)B_1}\right)\\
&\leq\frac{\Gamma^2(\beta_1+(k+1)B_1)}{\Gamma(\beta_1+kB_1)\Gamma(\beta_1+(k+2)B_1)},
\end{split}
\end{equation}
for $\alpha_i\geq\beta_{i+1},\;i=1,...,p.$ On the other hand, taking in (\ref{32}) the values $z=\beta_1+kB_1$ and $a=b=B_1$, we deduce the following Tur\'an type inequalities
\begin{equation}\label{4444}
\Gamma(\beta_1+kB_1)\Gamma(\beta_1+(k+2)B_1)-\Gamma^2(\beta_1+(k+1)B_1)\geq0.
\end{equation}
In view of (\ref{444}) and (\ref{4444}), we deduce that the sequence $(u_k)_{k\geq0}$ is decreasing. Thus,  the function 
$$z\mapsto \left({}_p\Psi_{p+1}^*\Big[_{(\beta_1, B_1),(\beta_p, 1)}^{\;\;\;\;(\alpha_p, 1)\;\;}\Big|z\Big]\right)^\prime\Big/{}_P\Psi_{p+1}^*\Big[_{(\beta_1, B_1),(\beta_p, 1)}^{\;\;\;\;(\alpha_p, 1)\;\;}\Big|z\Big]$$ is decreasing on $(0,\infty)$, and consequently the function $z\mapsto{}_p\Psi_{p+1}^*\Big[_{(\beta_1, B_1),(\beta_p, 1)}^{\;\;\;\;(\alpha_p, 1)\;\;}\Big|z\Big]$ is log-concave $(0,\infty).$ Thus implies that for all $t\in[0,1]$ and $z_1,z_2>0,$ we have $$
\left[{}_p\Psi_{p+1}^*\Big[_{(\beta_1, B_1),(\beta_p, 1)}^{\;\;\;\;(\alpha_p, 1)\;\;}\Big|z_1\Big]\right]^t\left[{}_p\Psi_{p+1}^*\Big[_{(\beta_1, B_1),(\beta_p, 1)}^{\;\;\;\;(\alpha_p, 1)\;\;}\Big|z_2\Big]\right]^{1-t}\leq {}_p\Psi_{p+1}^*\Big[_{(\beta_1, B_1),(\beta_p, 1)}^{\;\;\;\;(\alpha_p, 1)\;\;}\Big|tz_1+(1-t)z_2\Big],$$
setting $t=1/2$ we get the inequality (\ref{876}). Now let us focus on the inequality (\ref{877}), to prove this, let
$$f(z)=\log{}_p\Psi_{p+1}^*\Big[_{(\beta_1, B_1),(\beta_p, 1)}^{\;\;\;\;(\alpha_p, 1)\;\;}\Big|z\Big] \textrm{and}\;\; g(z)=z.$$ 
By using the fact that the function $f^\prime(z)$ is decreasing on  $(0,\infty),$ we deduce that the function $x\mapsto f(z)/g(z)=(f(z)-f(0))/(g(z)-g(0))$ is also decreasing on $(0,\infty).$ On the other hand, from the Bernouilli-l'Hospital's rule and the differentiation formula  (\ref{!!}), it is easy to deduce that
$$\lim_{x\rightarrow0}\frac{f(x)}{g(x)}=\left(\prod_{i=1}^p\frac{\alpha_i}{\beta_{i+1}}\right)\left(\frac{\Gamma(\beta_1)}{\Gamma(\beta_1+B_1)}\right).$$
Finally, for the proof of inequality (\ref{!!!}), we appeal again the monotonicity for the ratios $f^\prime(x)/g\prime(x)$, we get
$$f^\prime(x)\leq \left(\prod_{i=1}^p\frac{\alpha_i}{\beta_{i+1}}\right)\left(\frac{\Gamma(\beta_1)}{\Gamma(\beta_1+B_1)}\right).$$
By again the differentiation formula (\ref{!!}) the proof of inequality (\ref{!!!}) is done, which evidently completes the proof of Theorem  \ref{ttt34}.
\end{proof}

Taking in Theorem  \ref{ttt34} the value $B_1=1$, we obtain the following inequalities for the hypergeometric function ${}_pF_{p+1}.$

\begin{coro}Let $\alpha_1, \beta_1,\beta_2>0.$ If $\alpha_i\geq\beta_{i+1},\;i=1,...,p,$ then the function $z\mapsto{}_pF_{p+1}(z)$ is log-concave on $(0,\infty)$, and satisfies the following inequalities:
\begin{equation*}
\begin{split}
{}_pF_{p+1}\Big(^{\;\;\alpha_1,...,\alpha_p}_{\beta_1,...,\beta_{p+1}}\Big|z_1\Big){}_pF_{p+1}\Big(^{\;\;\alpha_1,...,\alpha_p}_{\beta_1,...,\beta_{p+1}}\Big|z_2\Big)&\leq {}_pF_{p+1}\Big(^{\;\;\alpha_1,...,\alpha_p}_{\beta_1,...,\beta_{p+1}}\Big|\frac{z_1+z_2}{2}\Big),\;z_1,z_2>0.\\
{}_pF_{p+1}\Big(^{\;\;\alpha_1,...,\alpha_p}_{\beta_1,...,\beta_{p+1}}\Big|z\Big)&\leq e^{\frac{\alpha_1...\alpha_p}{\beta_1...\beta_{p+1}}z},\;z>0.\\
{}_pF_{p+1}\Big(^{\;\;\alpha_1+1,...,\alpha_p+1}_{\beta_1+1,...,\beta_{p+1}+1}\Big|z_1\Big)&\leq\frac{\alpha_1...\alpha_p}{\beta_1...\beta_{p+1}}{}_pF_{p+1}\Big(^{\;\;\alpha_1,...,\alpha_p}_{\beta_1,...,\beta_{p+1}}\Big|z\Big) .
\end{split}
\end{equation*}
\end{coro} 

Next we show new inequalities  for the four-parametric Mittag--Leffler function $E_{B_1,\beta_1;1,\beta_2}(z)$ as follows.

\begin{coro}\label{c4} Let $\beta_1>0$ and $B_1\geq0.$ If $0<\beta_2\leq1,$ then the function $z\mapsto E_{\beta_1, B_1;\beta_2,1}(z)$ is log-concave on $(0,\infty).$ Moreover, the following inequalities  
\begin{equation}
\begin{split}
E_{B_1,\beta_1;1,\beta_2}(z_1)E_{B_1,\beta_1;1,\beta_2}(z_2)&\leq E_{B_1,\beta_1;1,\beta_2}((z_1+z_2)/2)\\
E_{B_1,\beta_1;1,\beta_2}(z)&\leq\frac{e^{\frac{\Gamma(\beta_1)z}{\beta_2\Gamma(\beta_1+B_1)}}}{\Gamma(\beta_1)},
\end{split}
\end{equation}
holds true.
\end{coro}
\begin{proof}Setting $\alpha_1=1$ in Theorem \ref{ttt34} we deduce that the function $z\mapsto E_{B_1,\beta_1;1,\beta_2}(z)$ is log-concave on $(0,\infty).$ This completes the proof of the two inequalities (20) asserted by  Corollary \ref{c4}. 
\end{proof}
\section{Open Problems}

Finally, motivated by the results of section 3 and Section 4, we pose the following problems:\\

\begin{problem} Proved the monotonicity of the function $\mathcal{K}_n^{(\alpha,\beta)}(A, B, z)$ defined in (\ref{RTR}).
\end{problem}
\begin{problem} Proved the monotonicity of the function $\Xi:(0,\infty)\longrightarrow\mathbb{R}$ defined
\begin{equation*}
\Xi(z)=\frac{\beta_1+B_1}{\beta_1}\log \Bigg[{}_p\tilde{\Psi}_q\Big[_{(\beta_1+1, B_1), (\beta_{q-1}, B_{q-1})}^{\;\;\;\;\;(\alpha_p, A_p)}\Big|z \Big]\Bigg]-\log \Bigg[{}_p\tilde{\Psi}_q\Big[_{(\beta_1, B_1), (\beta_{q-1}, B_{q-1})}^{\;\;\;\;\;(\alpha_p, A_p)}\Big|z \Big]\Bigg].
\end{equation*}
where
\begin{equation*}
\Xi^\prime(z)=\frac{\Gamma(\beta_1)}{\Gamma(\beta_1+B_1)}\left(\frac{{}_p\tilde{\Psi}_q\Big[_{(\beta_1+B_1+1,B_1), (\beta_{q-1}+B_{q-1},B_{q-1})}^{\;\;\;\;\;\;\;\;(\alpha_p+A_p, A_p)}\Big|z \Big]}{{}_p\tilde{\Psi}_q\Big[_{(\beta_1+1,B_1), (\beta_{q-1},B_{q-1})}^{\;\;\;\;\;\;\;\;(\alpha_p, A_p)}\Big|z \Big]}-\frac{{}_p\tilde{\Psi}_q\Big[_{(\beta_q+B_q,B_q)}^{(\alpha_p+A_p, A_p)}\Big|z \Big]}{{}_p\tilde{\Psi}_q\Big[_{(\beta_q, B_q)}^{(\alpha_p, A_p)}\Big|z \Big]}\right).
\end{equation*}
\end{problem}

\end{document}